\documentclass[12pt]{article}
\usepackage{abstract}
\usepackage{booktabs}
\usepackage{caption}
\usepackage{mathrsfs}
\usepackage{amsmath}
\usepackage{amsfonts,amsthm,amssymb,mathrsfs,bbding}
\usepackage{float}
\usepackage{graphics,epstopdf}
\usepackage{graphics,multicol}
\usepackage{graphicx}
\usepackage{color}
\usepackage{tikz}
\usepackage{enumerate}
\usepackage{caption}
\usepackage{rotating}
\usepackage{lscape}
\usepackage{longtable}
\allowdisplaybreaks[4]
\usepackage{tabularx}
\usepackage[colorlinks=true,anchorcolor=blue,filecolor=blue,linkcolor=blue,urlcolor=blue,citecolor=blue]{hyperref}
\usepackage{extarrows}
\usepackage{cite}
\usepackage{latexsym,bm}
\usepackage{mathtools}
\evensidemargin=\oddsidemargin\topmargin=-1.5cm

\newtheorem{thm}{Theorem}[section]

\newtheorem{lem}{Lemma}[section]
\newtheorem{cor}{Corollary}[section]

\newtheorem{conj}{Conjecture}[section]
\newtheorem{claim}{Claim}[section]

\addtocounter{section}{0}

\begin{document}
\title{\bf A spectral threshold for triangle counting\footnote{Supported by the National
Natural Science Foundation of China (No.\! 12571369).}}

\author{{\bf Yuhan Zhang},\quad
{\bf Mingqing Zhai}\thanks{Corresponding author. E-mail address: mqzhai@njust.edu.cn (M. Zhai).}
\\
{\footnotesize School of Mathematics and Statistics, Nanjing University of Science and Technology}\\
{\footnotesize Nanjing 210094, Jiangsu, China}\\
}
\date{}

\maketitle

\begin{abstract}
The 1970 spectral extension of Mantel's theorem, proved by Nosal,
states that every graph with $m$ edges and spectral radius
$\rho_1>\sqrt{m}$ contains at least one triangle.
Its quantitative refinement by Ning and Zhai later established that any graph $G$ with $m$ edges and spectral radius
$\rho_1\geq\sqrt{m}$ contains at least $\lfloor\frac{\sqrt{m}-1}{2}\rfloor$ triangles, unless $G$ is a complete bipartite graph.

In this paper, we further investigate the minimum number of triangles guaranteed under the strengthened spectral condition
$\rho_1\geq\sqrt{m}+c$,
where $c$ is a positive constant.
We prove that for any constant $c\in (0,\frac{1}{2}]$ and all sufficiently large $m$,
if $s=s(m)$ is a real-valued function satisfying $\lim_{m\to\infty} \frac{s}{m}=c$,
then every $m$-edge graph $G$ with spectral radius $\rho_1$ satisfying
$\rho_1^2\geq m-1+\frac{2s}{\rho_1-1}$ contains at least $s$ triangles. 
Moreover, we characterize the extremal graph achieving the minimal number of triangles.
In particular, when $s=\frac{m-1}2$, our result settles a conjecture proposed by Li, Feng, and Peng.

\end{abstract}

\begin{flushleft}

\end{flushleft}

\section{Introduction}\label{sec-1}

~~~~Mantel's theorem is one of the cornerstone results in extremal graph theory (see \cite{Mantel}).
This classical theorem precisely characterizes the maximum number of edges in an $n$-vertex triangle-free simple graph,
demonstrating that the complete bipartite graph \(K_{\lfloor n/2 \rfloor, \lceil n/2 \rceil}\)
is the unique extremal graph.

Rademacher (unpublished paper, see Erd\H{o}s ~\cite{E-1,E62}) showed that every \(n\)-vertex graph with \(e(T_{n,2}) + 1\) edges contains at least \(\lfloor n/2 \rfloor\) triangles, 
and Erd\H{o}s~\cite{E-3,E-4} later generalized this result: for a suitable range of \(k\),
every graph with \(\lfloor n^2/4 \rfloor + k\) edges contains at least \(k\lfloor n/2 \rfloor\) triangles.

Over the past century, this fundamental result has been extended and generalized in numerous directions (see \cite{bollobas2004extremal}).
A particularly powerful framework in these generalizations is the spectral approach,
which centers on the spectral radius of a graph, i.e., the largest eigenvalue of its adjacency matrix.
The spectral radius naturally connects the linear algebraic properties of a graph to its core structural invariants (see \cite{cvetkovic1995spectra}).
Since the adjacency matrix of an undirected graph is real symmetric,
it admits a complete set of real eigenvalues and an orthonormal eigenbasis,
laying a solid theoretical foundation for spectral graph theory.

A landmark spectral extension of Mantel's theorem was established by Nosal in 1970 (see \cite{N70}).
Prior to this work, nearly all extensions of Mantel's theorem relied on combinatorial arguments and vertex-degree conditions,
while the direct connection between the spectral radius and the existence of triangles remained largely unexamined.
Nosal's pioneering result filled this gap, proving that any graph with \(m\) edges and spectral radius \(\rho_1 > \sqrt{m}\) must contain at least one triangle.
This theorem was the first to establish a sharp, edge-count-based spectral condition for the presence of triangles,
opening a new direction in spectral extremal graph theory and inspiring decades of follow-up research on spectral conditions for subgraph existence and counting.
Building on this, Bollob\'as and Nikiforov~\cite{BN07} further refined this line of inquiry by deriving the inequality
\[
t(G) \geq \frac{\rho_1(G)\big(\rho_1^2(G) - m\big)}{3},
\]
which explicitly links the number \(t(G)\) of triangles to the spectral radius \(\rho_1(G)\) and the number \(m\) of edges.

Building on Bollob\'as and Nikiforov's foundational result, Ning and Zhai~\cite{M01} later provided a critical quantitative refinement.
They proved that any graph \(G\) with \(m\) edges and spectral radius \(\rho_1(G)\geq \sqrt{m}\) contains at least \(\lfloor \frac{\sqrt{m} - 1}{2} \rfloor\) triangles,
with the sole exception of complete bipartite graphs (which are triangle-free).
Recently,
Li, Feng, and Peng~\cite{M02} proposed the following conjecture, which further pushes the boundaries of spectral triangle counting:

\begin{conj}[Li, Feng and Peng~\cite{M02}]\label{conj-1}
Let \(G\) be a graph with \(m\) edges, where \(m\) is sufficiently large. If \(G\) satisfies
\[
\rho_1(G) \geq \frac{1 + \sqrt{4m - 3}}{2},
\]
then \(t(G) \geq \frac{m - 1}{2}\), with equality if and only if \(G = K_2 \nabla \frac{m - 1}{2}K_1\).
\end{conj}

Motivated by these breakthroughs and the growing interest in sharp spectral conditions for subgraph counting,
we further investigate the minimum number of triangles guaranteed by a strengthened spectral condition.
Specifically, we consider the condition \(\rho_1 \geq \sqrt{m} + c\) for a fixed positive constant \(c\),
which pushes the spectral radius strictly above the classical Nosal's condition.
This strengthened condition is designed to reveal how the guaranteed minimum number of triangles grows
as the spectral radius exceeds \(\sqrt{m}\) by a fixed constant---a question left fully open by Nosal's result and Ning and Zhai's bound.

For integer $s$,
let \( S^{+}_{m,s} \) denote the graph constructed by embedding an edge into the \(\frac{1}{s}(m - 1)\)-vertex partition of \( K_{s,\frac{m-1}s} \).
In particular, when \( s =\frac{m-1}2\),
\( S^{+}_{m,s} \) is clearly isomorphic to $K_2 \nabla \frac{m - 1}{2}K_1$.
We prove the following theorem.

\begin{thm}\label{thm-1}
Let $m$ be sufficiently large, and let $s=s(m)$ be any real-valued function
with $\lim\limits_{m\to\infty}\frac{s}{m}=c$, where $c\in\left(0,\frac{1}{2}\right]$.
If $G$ is an $m$-edge graph such that
$$\rho_1^2(G)\geq m-1+\frac{2s}{\rho_1(G)-1},$$
then $t(G)\geq s$. Equality holds if and only if $G\cong S_{m,s}^{+}$.
\end{thm}

Our proof relies on classical inequalities for adjacency eigenvalues, subgraph counting techniques, 
and constructions of induced subgraphs.
Note that
\[
\rho_1^2 \geq m - \frac{1}{2} + \sqrt{m - \frac{3}{4}} \iff \rho_1^2 \geq m - 1 + \frac{2s}{\rho_1 - 1}, \text{where } s = \frac{m - 1}{2}.
\]

As an important special case, setting \(s = \frac{m - 1}{2}\), we obtain the following corollary:
\begin{cor}\label{cor-1}
Let \(G\) be a graph with \(m\) edges, where \(m\) is sufficiently large. If
\[
\rho_1(G)\geq\frac{1+\sqrt{4m-3}}{2},
\]
then \(t(G) \geq \frac{m - 1}{2}\), with equality if and only if \(G\cong S^{+}_{m,\frac{m - 1}{2}}\).
\end{cor}

Thus, Corollary \ref{cor-1} confirms Conjecture \ref{conj-1}. Beyond resolving this conjecture,
our result generalizes both Nosal's original theorem and Ning and Zhai's quantitative refinement.


\section{Preliminary}\label{sec-2}

~~~~In this section, we introduce several necessary lemmas.

\begin{lem}\label{lem-2.1}(Ning and Zhai\cite{M01}).
Let $G$ be a graph on $n$ vertices and $m$ edges,
and let $\rho_{1}\geq\rho_{2}\geq\cdots\geq\rho_{n}$ be all adjacency eigenvalues of $G$. Then
\[
t(G)=\frac{1}{6}\sum_{i=2}^{n}(\rho_{1}+\rho_{i})\rho_{i}^{2}+\frac13\rho_{1}(\rho_{1}^{2}-m).
\]
\end{lem}

The following result is known as the Cauchy interlacing theorem (see \cite{cvetkovic1995spectra}).

\begin{lem}\label{lem-2.2}
Let $A$ be a symmetric $n\times n$ matrix and $B$ be an $r\times r$ principal submatrix of $A$ for some $r < n$.
If the eigenvalues of $A$ are $\rho_{1}\geq\rho_{2}\geq\ldots\geq\rho_{n}$
and the eigenvalues of $B$ are $\mu_{1}\geq\mu_{2}\geq\ldots\geq\mu_{r}$,
then $\rho_{i}\geq\mu_{i}\geq\rho_{n-r+i}$ holds for all $1\leq i\leq r$.
\end{lem}

The following lemma provides a useful edge-switching tool for $\rho_1(G)$.

\begin{lem}\label{lem-2.3}(Wu, Xiao, and Hong\cite{M03}).
Let $G$ be a connected graph with vertices $v_i,v_j\in V(G)$ and let $S$ be a nonempty subset of
$N_G(v_j)\setminus N_G(v_i)$. Assume that
\[
G^*=G-\{v_jv\mid v\in S\}+\{v_iv\mid v\in S\},
\]
and $(x_{v_1},\ldots,x_{v_n})^T$ is the Perron vector of $G$.
If $x_{v_i}\geq x_{v_j}$, then $\rho_{1}(G^*)>\rho_{1}(G)$.
\end{lem}

\section{Proof of Theorem \ref{thm-1}}\label{sec-3}

~~~~In this section, we present the proof of Theorem \ref{thm-1}.
Recall that $s$ is a real-valued function depending on $m$ and
satisfying $\lim\limits_{m\to\infty}\frac{s}{m}=c$, where $c\in\left(0,\frac{1}{2}\right]$ is a constant.
We now define a graph family as follows:
\[
\mathcal{G}=
\left\{
G \mid \rho_1^2\geq m-1+\frac{2s}{\rho_1-1},\ t(G)\leq s,\ e(G)=m
\right\}.
\]

If $\mathcal{G}$ is empty, then Theorem \ref{thm-1} holds immediately.
Therefore, we may assume that $\mathcal{G}$ is nonempty,
and we only need to check all graphs in $\mathcal{G}$.
Observe that for integer $s$, $S_{m,s}^+$ has exactly $s$ triangles and satisfies $\rho_1^2=m-1+\frac{2s}{\rho_1-1}$.
Let $G^*$ be an extremal graph in $\mathcal{G}$ that achieves the maximum spectral radius.
To complete the proof of Theorem \ref{thm-1},
it suffices to show that $s$ is an integer and $G^*\cong S_{m,s}^+$.

Without loss of generality,
we assume that the graph under consideration has no isolated vertices and $|G^*|=n$.
Denote $\rho_i:=\rho_i(G^*)$, where the eigenvalues are ordered such that
$\rho_1\geq \rho_2\geq\cdots\geq \rho_n$.
Then, we know that $\sum_{i=1}^n\rho_i^2=2m$ (see \cite{cvetkovic1995spectra}).
Let $X=(x_{1},x_{2},\ldots,x_{n})^{T}$ be the Perron vector of $G^*$,
and let $u^*\in V(G^*)$ be a vertex with the maximum Perron coordinate, that is,
$
x_{u^*}=\max_{u\in V(G^*)}x_u.
$

Denote $U=N(u^*)$ and $W=V(G^*)\setminus\bigl(U\cup\{u^*\}\bigr)$. Then, we have
\begin{align}\label{eq1}
\rho_1^2x_{u^*}
&=
d(u^*)x_{u^*}
+\sum_{uv\in E(U)}\!\!(x_u+x_v)
+\sum_{w\in W}\!d_W(w)x_w                                     \nonumber \\
&\leq
\bigl(d(u^*)+2e(U)+e(U,W)\bigr)x_{u^*} \nonumber\\
&=
\bigl(m+e(U)\bigr)x_{u^*}.
\end{align}

Given that $\lim_{m\to\infty}\frac{s}{m}=c$, we obtain the asymptotic relation
$s=\big(c\pm o(1)\big)m$ as $m$ tends to infinity. Substituting the above expression into the inequality for $\mathcal{G}$, we obtain
$
\rho_1^2\geq m-1+\frac{2}{\rho_1-1}\big(c-o(1)\big)m.
$
Therefore, we have
\begin{equation}\label{eq2}
\rho_1>\sqrt{m}.
\end{equation}
Note that $t(G^*)\leq s$. By Lemma \ref{lem-2.1}, we obtain
\begin{equation}\label{eq3}
\sum_{i=2}^{n}(\rho_1+\rho_i)\rho_i^2
+
2\rho_1(\rho_1^2-m)
\leq 6s.
\end{equation}
Moreover,
since $\sum_{i=2}^{n}(\rho_1+\rho_i)\rho_i^2\geq 0$ and $\lim\limits_{m\to\infty}\frac{s}{m}=c\leq\frac12$,
we have
$$
\rho_1(\rho_1^2-m)\leq 3s\leq 3\big(c+o(1)\big)m\leq \frac85m.
$$
Dividing both sides of the above equality by $\rho_{1}$ and using (\ref{eq2}), we derive that
\begin{equation}\label{eq4}
\rho_1^2-m
\leq
\frac{8m}{5\rho_1}
<
\frac{8m}{5\sqrt{m}}
=
\frac85\sqrt{m}.
\end{equation}
Combining (\ref{eq2}) and (\ref{eq4}) gives
$
\rho_1-\sqrt{m}
=
\frac{\rho_1^2-m}{\rho_1+\sqrt{m}}
<
\frac45.
$
It follows that
\begin{equation}\label{eq5}
\rho_1<\sqrt{m}+\frac45.
\end{equation}
Using (\ref{eq5}) and the asymptotic relation $s=\big(c-o(1)\big)m$, we have
\begin{equation}\label{eq6}
\frac{2s}{\rho_1-1}
\geq
\frac{2\big(c-o(1)\big)m}{\sqrt{m}}
=
2\big(c-o(1)\big)\sqrt{m}.
\end{equation}
Since $\rho_1^2\geq m-1+\frac{2s}{\rho_1-1}$, we have
\begin{align}\label{eq7}
\rho_1(\rho_1^2-m)
\geq
\rho_1\big(-1+\frac{2s}{\rho_1-1}\big)
=
2s\big(1+\frac{1}{\rho_1-1}\big)-\rho_1 .
\end{align}
Combining (\ref{eq3}), (\ref{eq5}) and (\ref{eq7}), we obtain
\begin{align}\label{eq8}
\sum_{i=2}^{n}(\rho_1+\rho_i)\rho_i^2
&\leq
2s-\frac{4s}{\rho_1-1}+2\rho_1       \notag
\leq
2s-\frac{4s}{\sqrt{m}}+2\big(\sqrt{m}+\frac45\big) \notag\\
&=
2s+\big(2-4c+o(1)\big)\sqrt{m}.
\end{align}

\begin{claim}\label{clm-3.1}
$\sum_{i=2}^{n-1}\rho_i^2=o(\sqrt{m})$
and $(\rho_1+\rho_n)\rho_n^2$ is increasing with respect to $\rho_n$.
\end{claim}

\begin{proof}
Recall that $2m=\sum_{i=1}^n\rho_i^2$ and $\rho_1^2\geq m-1+\frac{2s}{\rho_1-1}$.
Moreover, it is easy to see that $(\rho_1-\rho_n)^2\leq2(\rho_1^2+\rho_n^2)$. Thus,
\begin{align}\label{eq9}
\rho_1+\rho_n
&=
\frac{\rho_1^2-\rho_n^2}{\rho_1-\rho_n}                                      \notag
\geq
\frac{
\rho_1^2-\left(2m-\sum_{i=1}^{n-1}\rho_i^2\right)
}{
\sqrt{2(\rho_1^2+\rho_n^2)}
}                                                                            \notag
\geq
\frac{
2(\rho_1^2-m)+\sum_{i=2}^{n-1}\rho_i^2
}{2\sqrt{m}}                                                                            \notag\\
&\geq
\frac{1}{\sqrt{m}}
\Big(
\frac{2s}{\rho_1-1}
-1
+\frac12\sum_{i=2}^{n-1}\rho_i^2
\Big).
\end{align}

Suppose first that
$
\lim_{m\to\infty}
\frac{1}{m}\sum_{i=2}^{n-1}\rho_i^2
=
c_1\neq 0.
$
From (\ref{eq2}) we know that $\rho_1>\sqrt{m}$.
Consequently,
$\sum_{i=2}^{n-1}\rho_i^2\leq 2m-\rho_1^2\leq m$, which implies $c_1\leq 1$.
Furthermore,
$$
\rho_n^2
=
2m-\rho_1^2-\sum_{i=2}^{n-1}\rho_i^2
\leq
\big(1-c_1+o(1)\big)m
\leq
\big(1-\frac{1}{2}c_1\big)m.
$$
Hence, $\rho_n\geq -\sqrt{\big(1-\frac{1}{2}c_1\big)m}$.
It follows that
$
\rho_1+\rho_i
\geq
\rho_1+\rho_n
\geq
\big(1-\sqrt{1-\frac{1}{2}c_1}\big)\sqrt{m}$ for any $i\in\{2,\ldots,n-1\}$.
Since $\lim_{m\to\infty}
\frac{\sum_{i=2}^{n-1}\rho_i^2}{m}=c_1\leq1$, we have
\begin{align*}
\sum_{i=2}^{n\!-\!1}(\rho_1\!+\!\rho_i)\rho_i^2
\geq
\Big(1\!-\!\sqrt{1\!-\!\frac{1}{2}c_1}\Big)\sqrt{m}
\sum_{i=2}^{n\!-\!1}\rho_i^2
\geq
\Big(c_1\Big(1\!-\!\sqrt{1\!-\!\frac{1}{2}c_1}\Big)\!-\!o(1)\Big)m^{\frac32},
\end{align*}
which contradicts inequality (\ref{eq8}).

We next suppose that
$\lim_{m\to\infty}
\frac{1}{m^\alpha}\sum_{i=2}^{n-1}\rho_i^2=c_2\neq 0$
for some fixed $\alpha$ satisfying $\frac12\leq \alpha<1$.
Based on (\ref{eq5}), we have $\rho_1<\sqrt{m}+\frac45$. 
Since $\alpha<1$, it follows that
$$\rho_n^2=2m-\rho_1^2-\sum_{i=2}^{n-1}\rho_i^2\geq \big(1-o(1)\big)m.$$
Substituting the estimate from (\ref{eq6}) into the expression for $\rho_{1}+\rho_{n}$ in (\ref{eq9}),
we obtain
\begin{align*}
\rho_1+\rho_n
\geq
\frac{1}{\sqrt{m}}
\Big(
2\big(c-o(1)\big)\sqrt{m}
-1+\frac12\big(c_2-o(1)\big)m^\alpha
\Big)\geq\frac{c_2}{4}m^{\alpha-\frac12}.
\end{align*}
It follows that
$
(\rho_1+\rho_n)\rho_n^2
\geq
\frac{c_2}{4}\big(1-o(1)\big)m^{\alpha+\frac12}.
$
We now consider the two admissible cases for $\alpha$ separately:

When $\alpha>\frac12$, the lower bound on $(\rho_1+\rho_n)\rho_n^2$ grows superlinearly in $m$.
This directly contradicts the linear upper bound for the same quantity given in \eqref{eq8}.

When $\alpha = \frac{1}{2}$, using \eqref{eq6}, \eqref{eq9}, and the bound $\rho_n^2\geq \big(1-o(1)\big)m$, we obtain
\begin{align}\label{eq10}
(\rho_1+\rho_n)\rho_n^2
&\geq
\big(2c+\frac{c_2}{2}-o(1)\big)\big(1-o(1)\big)m =\big(2c+\frac{c_2}{2}-o(1)\big)m.
\end{align}
On the other hand,
inequality \eqref{eq8} gives the following upper bound:
$
\sum_{i=2}^{n}(\rho_1+\rho_i)\rho_i^2 \leq 2s + o(m) = \big(2c + o(1)\big)m.
$
This contradicts the lower bound in \eqref{eq10}.

Now, we conclude that
$\sum_{i=2}^{n-1}\rho_i^2=o(\sqrt{m}).$
This implies that $|\rho_i|=o(m^{\frac14})$ for $i\in\{2,\ldots,n-1\}$.
Combining \eqref{eq2} gives $\rho_1+\rho_i\geq \big(1-o(1)\big)\sqrt{m}$ for $2\leq i\leq n-1$.
Moreover, in view of (\ref{eq4}), we know that
$\rho_1^2\leq m+\frac85\sqrt{m}$ and thus
$$\rho_n^2=2m-\rho_1^2-\sum_{i=2}^{n-1}\rho_i^2\geq m-2\sqrt{m}\geq\frac49\rho_1^2.$$
It follows that $\rho_n\leq -\frac23\rho_1<0$.

To investigate the monotonicity of $(\rho_1+\rho_n)\rho_n^2$,
we define $f(\rho_n)=(\rho_1+\rho_n)\rho_n^2$.
Differentiating $f(\rho_n)$ yields:
$$
f'(\rho_n)
=
\rho_n^2+2\rho_n(\rho_1+\rho_n)
=
\rho_n(3\rho_n+2\rho_1)>0.
$$
It follows immediately that $(\rho_1+\rho_n)\rho_n^2$ is increasing with respect to $\rho_n$.
\end{proof}

\begin{claim}\label{clm-3.2}
For all $c\in(0,\frac12]$, the graph $G^*$ contains no induced subgraph $H$ such that
$\sum_{i=2}^{|H|-1}\rho_i^2(H)\geq 2.$
Furthermore, if $c\in(\frac14,\frac12]$, then $G^*$ contains no induced subgraph $H$ such that
$\sum_{i=2}^{|H|-1}\rho_i^2(H)\geq \frac32.$
\end{claim}

\begin{proof}
Suppose, for the sake of contradiction, that $G^*$ contains an induced subgraph $H$ satisfying
$\sum_{i=2}^{|H|-1}\rho_i^2(H)\geq 2.$
It is known that $\sum_{i=1}^{|H|}\rho_i(H)=0,$
which implies that there exists $i_0\in\{2,\ldots,|H|\!-\!1\}$ such that $\rho_{i_0}(H)<0$ and
$\rho_{i_0-1}(H)\geq0$.

By Lemma \ref{lem-2.2}, we have the following interlacing inequalities:
$\rho_i\geq \rho_i(H)\geq0$ for $i\leq i_0-1$, and $\rho_{n-|H|+i}\leq \rho_i(H)<0$ for $i\geq i_0$.
It follows that
\begin{equation}\label{eq11}
\sum_{i=2}^{n-1}\rho_i^2\geq\sum_{i=2}^{i_0-1}\rho_i^2+\sum_{i=i_0}^{|H|-1}\rho_{n-|H|+i}^2\geq\sum_{i=2}^{|H|-1}\rho_i^2(H)\geq2.
\end{equation}
By Claim \ref{clm-3.1}, $\sum_{i=2}^{n-1}\rho_i^2=o(\sqrt{m}).$
Combining \eqref{eq2} gives $\rho_1+\rho_i\geq\big(1-o(1)\big)\sqrt{m}$ for all $i\in\{2,\ldots,n-1\}$.
Consequently,
\begin{equation}\label{eq12}
\sum_{i=2}^{n-1}(\rho_1+\rho_i)\rho_i^2
\geq
\big(1-o(1)\big)\sqrt{m}\sum_{i=2}^{n-1}\rho_i^2
\geq
\big(2-o(1)\big)\sqrt{m}.
\end{equation}
Recall that $\rho_1^2\geq m-1+\frac{2s}{\rho_1-1}$.
From (\ref{eq11}), we estimate the smallest eigenvalue:
\begin{align*}
\rho_n^2
=
2m-\rho_1^2-\sum_{i=2}^{n-1}\rho_i^2
\leq
2m\!-\!\Big(m\!-\!1\!+\!\frac{2s}{\rho_1\!-\!1}\Big)\!-\!2
=m\!-\!1\!-\!\frac{2s}{\rho_1\!-\!1}.
\end{align*}
Hence,
$
\rho_n\geq-\sqrt{m-1-\frac{2s}{\rho_1-1}}.
$
Moreover, by Claim \ref{clm-3.1}, $(\rho_1+\rho_n)\rho_n^2$ is increasing with respect to $\rho_n$.
Thus, we obtain
\begin{align}
(\rho_1+\rho_n)\rho_n^2
&\geq
\Big(
\sqrt{m\!-\!1\!+\!\frac{2s}{\rho_1\!-\!1}}
\!-\!
\sqrt{m\!-\!1\!-\!\frac{2s}{\rho_1\!-\!1}}
\Big)
\Big(
m\!-\!1\!-\!\frac{2s}{\rho_1\!-\!1}
\Big)                                     \notag\\
&=
\frac{4s}{\rho_1\!-\!1}\
\Big(m\!-\!1\!-\!\frac{2s}{\rho_1\!-\!1}\Big)\Big/
\Big(\sqrt{m\!-\!1\!+\!\frac{2s}{\rho_1\!-\!1}}
+
\sqrt{m\!-\!1\!-\!\frac{2s}{\rho_1\!-\!1}}\Big)    \notag\\
&\geq
\frac{1}{2\sqrt{m}} \frac{4s}{\rho_1\!-\!1}\
\big(m\!-\!1\!-\!\frac{2s}{\rho_1\!-\!1}\big)     \notag\\
&=
\frac{2s\sqrt{m}}{\rho_1\!-\!1}
\!-\!
\frac{2s}{(\rho_1\!-\!1)\sqrt{m}}
\!-\!
\frac{4s^2}{(\rho_1\!-\!1)^2\sqrt{m}}.        \notag
\end{align}
Since $\sqrt{m}<\rho_1<\sqrt{m}+\frac45$ and $s=\big(c\pm o(1)\big)m$, we derive that
\begin{align}\label{eq13}
(\rho_1+\rho_n)\rho_n^2
\geq2s-\Big(4c^2+o(1)\Big)\sqrt{m}.
\end{align}
Summing inequalities (\ref{eq12}) and (\ref{eq13}), we obtain
$$\sum_{i=2}^{n}(\rho_1+\rho_i)\rho_i^2
\geq
2s+\Big(2-4c^2-o(1)\Big)\sqrt{m}.$$
For all constant $c\in(0,\frac12],$ we have $2-4c^2>2-4c$.
Hence, the lower bound obtained above contradicts the upper bound stated in (\ref{eq8}).

Now, let $c\in(\frac14,\frac12]$, and suppose that $G^*$ contains an induced subgraph $H$ with
$\sum_{i=2}^{|H|-1}\rho_i^2(H)\geq \frac32.$
By the same arguments as in (\ref{eq12}-\ref{eq13}),
we can derive that
\begin{equation}\label{eq14}
\sum_{i=2}^{n-1}(\rho_1+\rho_i)\rho_i^2
\geq
\big(1-o(1)\big)\sqrt{m}\sum_{i=2}^{n-1}\rho_i^2
\geq
\big(\frac32-o(1)\big)\sqrt{m}.
\end{equation}
and
\begin{align}\label{eq15}
(\rho_1+\rho_n)\rho_n^2
\geq2s-\Big(4c^2+\frac14+o(1)\Big)\sqrt{m}.
\end{align}
Summing inequalities (\ref{eq14}) and (\ref{eq15}), we obtain
$$\sum_{i=2}^{n}(\rho_1+\rho_i)\rho_i^2
\geq2s+\Big(\frac54-4c^2-o(1)\Big)\sqrt{m}.$$
For all constant $c\in(\frac14,\frac12],$ we have $(\frac14-c)(\frac34-c)<0$,
and thus $\frac54-4c^2>2-4c$.
Therefore, the lower bound obtained above contradicts the upper bound stated in (\ref{eq8}).
This contradiction forces the claim to hold.
\end{proof}

\begin{claim}\label{clm-3.3}
Denote $U_0:=\{u\in U \mid d_U(u)=0\}$ and $U_1=U\setminus U_0$.
Then, $e(U_1)\geq2\big(c-o(1)\big)\sqrt{m}$ and $G^*[U_1]$ is a connected triangle-free graph.
\end{claim}

\begin{proof}
Note that $e(U_1)=e(U)$. In view of (\ref{eq1}), we obtain
$$
\rho_1^2x_{u^*}
\leq
\bigl(d(u^*)+2e(U_1)+e(U,W)\bigr)x_{u^*}
\leq
\bigl(m+e(U_1)\bigr)x_{u^*}.
$$
Since $\rho_1^2\geq m-1+\frac{2s}{\rho_1-1},$
it follows from (\ref{eq6}) that
$$e(U_1)\geq\rho_1^2-m\geq -1+\frac{2s}{\rho_1-1}\geq2\big(c-o(1)\big)\sqrt{m}.$$

Observe that the graph $2K_2$ has eigenvalues $(1,1,-1,-1)$ and $K_4$ has eigenvalues $(3,-1,-1,-1)$.
By Claim \ref{clm-3.2}, $G^*$ admits no induced copy of $2K_2$ or $K_4$.
Consequently, the subgraph induced by $U_1$ is connected and triangle-free.
\end{proof}

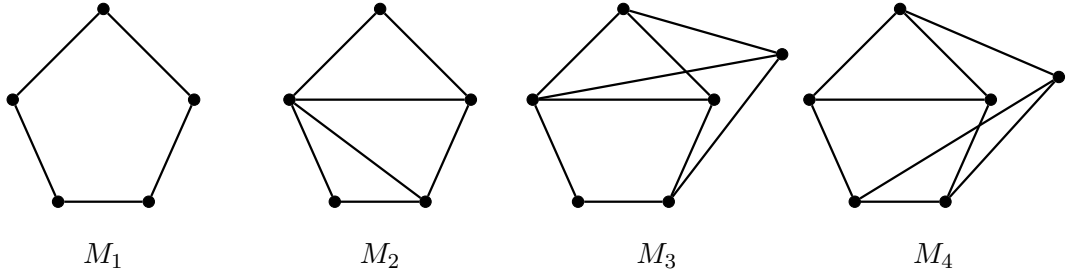
\begin{figure}[htbp]
\centering
\setlength{\tabcolsep}{3pt}

\begin{tabular}{cccc}

\parbox[c][4.8cm][c]{0.23\textwidth}{%
\centering
\begin{tikzpicture}[scale=3, x=1mm, y=1mm]
\node[circle,fill=black,draw=black,inner sep=0pt,minimum size=1.5mm] (m11) at (0,4.5) {};
\node[circle,fill=black,draw=black,inner sep=0pt,minimum size=1.5mm] (m12) at (4,0.5) {};
\node[circle,fill=black,draw=black,inner sep=0pt,minimum size=1.5mm] (m13) at (2,-4) {};
\node[circle,fill=black,draw=black,inner sep=0pt,minimum size=1.5mm] (m14) at (-2,-4) {};
\node[circle,fill=black,draw=black,inner sep=0pt,minimum size=1.5mm] (m15) at (-4,0.5) {};

\draw[line width=0.3mm, black] (m11) -- (m12);
\draw[line width=0.3mm, black] (m12) -- (m13);
\draw[line width=0.3mm, black] (m13) -- (m14);
\draw[line width=0.3mm, black] (m14) -- (m15);
\draw[line width=0.3mm, black] (m15) -- (m11);
\end{tikzpicture}
}
&
\parbox[c][4.8cm][c]{0.23\textwidth}{%
\centering
\begin{tikzpicture}[scale=3, x=1mm, y=1mm]
\node[circle,fill=black,draw=black,inner sep=0pt,minimum size=1.5mm] (m11) at (0,4.5) {};
\node[circle,fill=black,draw=black,inner sep=0pt,minimum size=1.5mm] (m12) at (4,0.5) {};
\node[circle,fill=black,draw=black,inner sep=0pt,minimum size=1.5mm] (m13) at (2,-4) {};
\node[circle,fill=black,draw=black,inner sep=0pt,minimum size=1.5mm] (m14) at (-2,-4) {};
\node[circle,fill=black,draw=black,inner sep=0pt,minimum size=1.5mm] (m15) at (-4,0.5) {};

\draw[line width=0.3mm, black] (m11) -- (m12);
\draw[line width=0.3mm, black] (m12) -- (m13);
\draw[line width=0.3mm, black] (m13) -- (m14);
\draw[line width=0.3mm, black] (m14) -- (m15);
\draw[line width=0.3mm, black] (m15) -- (m11);
\draw[line width=0.3mm, black] (m15) -- (m12);
\draw[line width=0.3mm, black] (m15) -- (m13);
\end{tikzpicture}
}
&
\parbox[c][4.8cm][c]{0.23\textwidth}{%
\centering
\begin{tikzpicture}[scale=3, x=1mm, y=1mm]
\node[circle,fill=black,draw=black,inner sep=0pt,minimum size=1.5mm] (n11) at (0,4.5) {};
\node[circle,fill=black,draw=black,inner sep=0pt,minimum size=1.5mm] (n12) at (4,0.5) {};
\node[circle,fill=black,draw=black,inner sep=0pt,minimum size=1.5mm] (n13) at (2,-4) {};
\node[circle,fill=black,draw=black,inner sep=0pt,minimum size=1.5mm] (n14) at (-2,-4) {};
\node[circle,fill=black,draw=black,inner sep=0pt,minimum size=1.5mm] (n15) at (-4,0.5) {};
\node[circle,fill=black,draw=black,inner sep=0pt,minimum size=1.5mm] (n16) at (7,2.5) {};

\draw[line width=0.3mm, black] (n11) -- (n12);
\draw[line width=0.3mm, black] (n12) -- (n13);
\draw[line width=0.3mm, black] (n13) -- (n14);
\draw[line width=0.3mm, black] (n14) -- (n15);
\draw[line width=0.3mm, black] (n15) -- (n11);
\draw[line width=0.3mm, black] (n16) -- (n11);
\draw[line width=0.3mm, black] (n16) -- (n15);
\draw[line width=0.3mm, black] (n16) -- (n13);
\draw[line width=0.3mm, black] (n15) -- (n12);
\end{tikzpicture}
}
&
\parbox[c][4.8cm][c]{0.23\textwidth}{%
\centering
\begin{tikzpicture}[scale=3, x=1mm, y=1mm]
\node[circle,fill=black,draw=black,inner sep=0pt,minimum size=1.5mm] (s11) at (0,4.5) {};
\node[circle,fill=black,draw=black,inner sep=0pt,minimum size=1.5mm] (s12) at (4,0.5) {};
\node[circle,fill=black,draw=black,inner sep=0pt,minimum size=1.5mm] (s13) at (2,-4) {};
\node[circle,fill=black,draw=black,inner sep=0pt,minimum size=1.5mm] (s14) at (-2,-4) {};
\node[circle,fill=black,draw=black,inner sep=0pt,minimum size=1.5mm] (s15) at (-4,0.5) {};
\node[circle,fill=black,draw=black,inner sep=0pt,minimum size=1.5mm] (s16) at (7,1.5) {};

\draw[line width=0.3mm, black] (s11) -- (s12);
\draw[line width=0.3mm, black] (s12) -- (s13);
\draw[line width=0.3mm, black] (s13) -- (s14);
\draw[line width=0.3mm, black] (s14) -- (s15);
\draw[line width=0.3mm, black] (s15) -- (s11);
\draw[line width=0.3mm, black] (s16) -- (s11);
\draw[line width=0.3mm, black] (s16) -- (s14);
\draw[line width=0.3mm, black] (s16) -- (s13);
\draw[line width=0.3mm, black] (s15) -- (s12);
\end{tikzpicture}
}

\\[-7mm]
\makebox[0.23\textwidth][c]{\small\strut $M_{1}$}
&
\makebox[0.23\textwidth][c]{\small\strut $M_{2}$}
&
\makebox[0.23\textwidth][c]{\small\strut $M_{3}$}
&
\makebox[0.23\textwidth][c]{\small\strut $M_{4}$}
\end{tabular}
\caption{The forbidden induced subgraphs $M_1$, $M_2$, $M_3$, and $M_4$.}
\label{fig-M1234}
\end{figure}

\begin{table}[htbp]
\caption{The intermediate eigenvalues of $M_1$, $M_2$, $M_3$, and $M_4$.}
\label{tab-M1234}
\begin{tabular*}{\textwidth}{@{\extracolsep{\fill}}ccccc@{}}
\toprule
& $M_1$ & $M_2$ & $M_3$ & $M_4$\\
\midrule
$\rho_2$ & 0.6180 & 0.6180 & 0.7020 & 1.0000 \\
$\rho_3$ & 0.6180 & -0.4626& 0.0000 & 0.0000 \\
$\rho_4$ & -1.6180 & -1.4728& 0.0000 & 0.0000 \\
$\rho_5$ & -- & -- & -1.2855 & -2.0000 \\
\bottomrule
\end{tabular*}
\end{table}

\begin{claim}\label{clm-3.4}
$e(W)=0$.
\end{claim}

\begin{proof}
Suppose that there exist $w_1,w_2\in W$ such that $w_1w_2\in E(W)$. Then
\begin{equation}\label{eq16}
N_U(w_1)\cup N_U(w_2)=U.
\end{equation}
If (\ref{eq16}) fails, then there exists a vertex $u \in U$ that is not adjacent to both $w_1$ and $w_2$.
In this case,
the subgraph induced by $\{u^*, u, w_1, w_2\}$ is isomorphic to $2K_2,$
yielding a contradiction. Hence, (\ref{eq16}) holds.

Still under the above assumption,
the given condition $x_{u^*} \geq \max\{x_{w_1}, x_{w_2}\}$ leads to the following degree bounds for $w_1$ and $w_2$ within $U$:
\begin{equation}\label{eq17}
d_U(w_i) \leq |U| - 1 \quad \text{for each } i \in \{1,2\}.
\end{equation}
Let $U'=N_U(w_1)\setminus N_U(w_2)$ and $U''=N_U(w_2)\setminus N_U(w_1)$.
Based on (\ref{eq16}) and (\ref{eq17}),
we know that both $U'$ and $U''$ are nonempty.

Take arbitrary vertices $u_1 \in U'$ and $u_2 \in U''.$
If $u_1u_2 \notin E(U)$,
then the subgraph induced by $\{u^*, u_1, u_2, w_1, w_2\}$ is isomorphic to $M_1$ (see Fig.\,\ref{fig-M1234}),
which contradicts Claim \ref{clm-3.2}.
Hence, $e(U',U'')=|U'||U''|$.
If there exists the third vertex $u_3\in U'\cup U''$,
then the subgraph induced by $\{u^*, u_1, u_2, u_3, w_1, w_2\}$ is isomorphic to $M_3$ (see Fig.\,\ref{fig-M1234}),
again contradicting Claim \ref{clm-3.2}.
Therefore, we must have $|U'| = |U''| = 1$.

Now, from the covering property established in (\ref{eq16}),
we know that
$U \setminus \{u_1, u_2\} = N_U(w_1) \cap N_U(w_2).$
Since $\rho_1x_{u^*}=\sum_{u\in U}x_u\leq |U|x_{u^*}$,
we have $|U|\geq \rho_1$.
In view of (\ref{eq2}), we obtain $|U|\geq \rho_1>\sqrt{m}$.
Pick $u_3 \in U \setminus \{u_1, u_2\}.$
By Claim \ref{clm-3.3}, $G^*[U_1]$ is triangle-free, which implies that $\{u_1,u_2\}\nsubseteq N_U(u_3)$.
If $u_1,u_2\notin N_U(u_3)$,
then the subgraph induced by $\{u^*, u_1, u_2, u_3, w_1, w_2\}$ is isomorphic to $M_4$ (see Fig.\,\ref{fig-M1234}),
contradicting Claim \ref{clm-3.2}.
Hence, we may assume that $u_1\in N_U(u_3)$ and  $u_2\notin N_U(u_3)$.
But now, the join of $u_3$ and $u^*u_1w_1w_2$ is a copy of $M_2$ (see Fig.\,\ref{fig-M1234}),
also a contradiction.
Therefore, the initial assumption $e(W) \geq 1$ is false, and we conclude that $e(W) = 0$.
\end{proof}

\begin{claim}\label{clm-3.5}
$G^*[U_1]$ is a complete bipartite graph.
Moreover, if $G^*[U_1]$ is not a star, then $U_0=\emptyset$ and $W\neq\emptyset$.
\end{claim}

\begin{proof}
If $G^*[U_1]$ is a star, we are done.
Now, assume that $G^*[U_1]$ is not a star.
From Claim \ref{clm-3.3}, we know that $G^*[U_1]$ is connected and triangle-free.
Hence, $G^*[U_{1}]$ contain either $P_{4}$ or $C_{4}$ as an induced subgraph.
If $G^*[U_{1}]$ contains an induced copy of $P_4$,
then $G^*$ contains $M_2$ as an induced subgraph (see Fig.\,\ref{fig-M1234}), which contradicts Claim \ref{clm-3.2}.
Therefore, $G^*[U_{1}]$ must contain an induced subgraph $H_0$, which is a 4-cycle.
Since $G^*$ cannot contain an induced copy of $M_6$ (see Fig.\,\ref{fig-M5678}),
we have $d_{V(H_0)}(u)\neq0$ for any $u\in U\setminus V(H_0)$.
Therefore, we conclude that $U_0=\emptyset$.

\vspace{-6mm}
\begin{figure}[htbp]
\centering
\setlength{\tabcolsep}{3pt}

\begin{tabular}{cccc}

\parbox[c][4.8cm][c]{0.23\textwidth}{%
\centering
\begin{tikzpicture}[scale=2.6, x=1mm, y=1mm]
\node[circle,fill=black,draw=black,inner sep=0pt,minimum size=1.5mm] (n41) at (0,5) {};
\node[circle,fill=black,draw=black,inner sep=0pt,minimum size=1.5mm] (n42) at (3,0) {};
\node[circle,fill=black,draw=black,inner sep=0pt,minimum size=1.5mm] (n43) at (0,-5) {};
\node[circle,fill=black,draw=black,inner sep=0pt,minimum size=1.5mm] (n44) at (-3,0) {};
\node[circle,fill=black,draw=black,inner sep=0pt,minimum size=1.5mm] (n45) at (6,1.8) {};
\node[circle,fill=black,draw=black,inner sep=0pt,minimum size=1.5mm] (n46) at (6,0) {};

\draw[line width=0.3mm, black] (n44) -- (n41);
\draw[line width=0.3mm, black] (n41) -- (n42);
\draw[line width=0.3mm, black] (n42) -- (n43);
\draw[line width=0.3mm, black] (n43) -- (n44);
\draw[line width=0.3mm, black] (n44) -- (n42);
\draw[line width=0.3mm, black] (n42) -- (n45);
\draw[line width=0.3mm, black] (n42) -- (n46);
\end{tikzpicture}
}
&
\parbox[c][4.8cm][c]{0.23\textwidth}{%
\centering
\begin{tikzpicture}[scale=2.6, x=1mm, y=1mm]
\node[circle,fill=black,draw=black,inner sep=0pt,minimum size=1.5mm] (m51) at (0,5) {};
\node[circle,fill=black,draw=black,inner sep=0pt,minimum size=1.5mm] (m52) at (-4,0) {};
\node[circle,fill=black,draw=black,inner sep=0pt,minimum size=1.5mm] (m53) at (0,0) {};
\node[circle,fill=black,draw=black,inner sep=0pt,minimum size=1.5mm] (m54) at (4,0) {};
\node[circle,fill=black,draw=black,inner sep=0pt,minimum size=1.5mm] (m55) at (0,-5) {};
\node[circle,fill=black,draw=black,inner sep=0pt,minimum size=1.5mm] (m56) at (6,4) {};

\draw[line width=0.3mm, black] (m51) -- (m52);
\draw[line width=0.3mm, black] (m52) -- (m55);
\draw[line width=0.3mm, black] (m55) -- (m54);
\draw[line width=0.3mm, black] (m54) -- (m51);
\draw[line width=0.3mm, black] (m51) -- (m53);
\draw[line width=0.3mm, black] (m55) -- (m53);
\draw[line width=0.3mm, black] (m52) -- (m53);
\draw[line width=0.3mm, black] (m53) -- (m54);
\draw[line width=0.3mm, black] (m56) -- (m53);
\end{tikzpicture}
}
&
\parbox[c][4.8cm][c]{0.23\textwidth}{%
\centering
\begin{tikzpicture}[scale=2.6, x=1mm, y=1mm]
\node[circle,fill=black,draw=black,inner sep=0pt,minimum size=1.5mm] (m61) at (0,5) {};
\node[circle,fill=black,draw=black,inner sep=0pt,minimum size=1.5mm] (m62) at (-4,0) {};
\node[circle,fill=black,draw=black,inner sep=0pt,minimum size=1.5mm] (m63) at (0,0) {};
\node[circle,fill=black,draw=black,inner sep=0pt,minimum size=1.5mm] (m64) at (4,0) {};
\node[circle,fill=black,draw=black,inner sep=0pt,minimum size=1.5mm] (m65) at (0,-5) {};
\node[circle,fill=black,draw=black,inner sep=0pt,minimum size=1.5mm] (m66) at (7,1) {};

\draw[line width=0.3mm, black] (m61) -- (m62);
\draw[line width=0.3mm, black] (m62) -- (m65);
\draw[line width=0.3mm, black] (m65) -- (m64);
\draw[line width=0.3mm, black] (m64) -- (m61);
\draw[line width=0.3mm, black] (m61) -- (m63);
\draw[line width=0.3mm, black] (m65) -- (m63);
\draw[line width=0.3mm, black] (m62) -- (m63);
\draw[line width=0.3mm, black] (m63) -- (m64);
\draw[line width=0.3mm, black] (m66) -- (m61);
\end{tikzpicture}
}
&
\parbox[c][4.8cm][c]{0.23\textwidth}{%
\centering
\begin{tikzpicture}[scale=2.6, x=1mm, y=1mm]
\node[circle,fill=black,draw=black,inner sep=0pt,minimum size=1.5mm] (m71) at (0,5) {};
\node[circle,fill=black,draw=black,inner sep=0pt,minimum size=1.5mm] (m72) at (-4,0) {};
\node[circle,fill=black,draw=black,inner sep=0pt,minimum size=1.5mm] (m73) at (0,0) {};
\node[circle,fill=black,draw=black,inner sep=0pt,minimum size=1.5mm] (m74) at (4,0) {};
\node[circle,fill=black,draw=black,inner sep=0pt,minimum size=1.5mm] (m75) at (0,-5) {};
\node[circle,fill=black,draw=black,inner sep=0pt,minimum size=1.5mm] (m76) at (6,0) {};
\node[circle,fill=black,draw=black,inner sep=0pt,minimum size=1.5mm] (m77) at (8,0) {};

\draw[line width=0.3mm, black] (m71) -- (m72);
\draw[line width=0.3mm, black] (m72) -- (m75);
\draw[line width=0.3mm, black] (m75) -- (m74);
\draw[line width=0.3mm, black] (m74) -- (m71);
\draw[line width=0.3mm, black] (m71) -- (m73);
\draw[line width=0.3mm, black] (m75) -- (m73);
\draw[line width=0.3mm, black] (m72) -- (m73);
\draw[line width=0.3mm, black] (m73) -- (m74);
\draw[line width=0.3mm, black] (m71) -- (m76);
\draw[line width=0.3mm, black] (m75) -- (m76);
\draw[line width=0.3mm, black] (m71) -- (m77);
\draw[line width=0.3mm, black] (m75) -- (m77);
\end{tikzpicture}
}
\\[-7mm]
\makebox[0.23\textwidth][c]{\small\strut $M_{5}$}
&
\makebox[0.23\textwidth][c]{\small\strut $M_{6}$}
&
\makebox[0.23\textwidth][c]{\small\strut $M_{7}$}
&
\makebox[0.23\textwidth][c]{\small\strut $M_{8}$}
\end{tabular}
\caption{The forbidden induced subgraphs $M_5$, $M_6$, $M_7$, and $M_8$.}
\label{fig-M5678}
\end{figure}
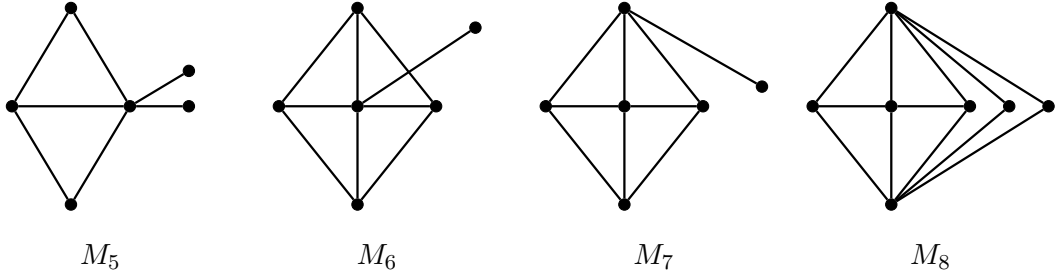

\vspace{-4mm}
\begin{table}[htbp]
\caption{The intermediate eigenvalues of $M_5$, $M_6$, $M_7$, and $M_8$.}
\label{tab-M5678}
\begin{tabular*}{\textwidth}{@{\extracolsep{\fill}}ccccc@{}}
\toprule
& $M_5$ & $M_6$ & $M_7$ & $M_8$ \\
\midrule
$\rho_2$ & 0.5293 & 0.3579 & 0.7347 & 0.5530 \\
$\rho_3$ & 0.0000 & 0.0000 & 0.0000 & 0.0000 \\
$\rho_4$ & 0.0000 & 0.0000 & -0.5975 & 0.0000 \\
$\rho_5$ & -1.3429 & -1.6813 & -1.2927 & 0.0000 \\
$\rho_6$ & -- & -- & -- & -1.3673 \\
\bottomrule
\end{tabular*}
\end{table}

Let $H_0=u_1u_2u_3u_4u_1$.
If there exists $u\in U\setminus V(H_0)$ such that $d_{V(H_0)}(u)=1$,
say $uu_1\in E(G^*)$,
then the subgraph induced by $\{u^*,u_0,u_1,u_2,u_3\}$ is isomorphic to $M_4$,
a contradiction.
Hence, we have $d_{V(H_0)}(u)\geq2$ for each $u\in U\setminus V(H_0)$.
Since $G^*[U]$ is triangle-free, we conclude that for any $u\in U\setminus V(H_0)$, either
$N_U(u)=\{u_1,u_3\}$ or $N_U(u)=\{u_2,u_4\}.$
It follows that
$G^*[U]$ is complete bipartite.

Let $U'\cup U''$ be the bipartition of $G^*[U]$, where $|U'|=a\geq2$ and $|U''|=b\geq2$.
Since $a+b=|U|$, we have
\begin{equation}\label{eq18}
ab\geq 2(|U|-2).
\end{equation}
Note that $ab=e(U)\leq t(G^*)\leq s.$
Combining (\ref{eq18}) gives
\begin{equation}\label{eq19}
|U|\leq \frac{s}{2}+2.
\end{equation}
Recall that $s=\big(c\pm o(1)\big)m$. Then,
\begin{equation}\label{eq20}
e\big(U\cup\{u^*\}\big)=|U|+ab\leq\frac32s+2\leq
\big(\frac32c+o(1)\big)m.
\end{equation}
Since $c\leq\frac12$, we have $e\big(U\cup\{u^*\}\big)<m$.
Hence, $W\neq\emptyset$.
\end{proof}

\begin{claim}\label{clm-3.6}
$G^*[U_1]$ is a star.
\end{claim}

\begin{proof}
Suppose, to the contrary, that $G^*[U_1]$ is not a star.
By Claim \ref{clm-3.5}, $U_0=\emptyset$, $W\neq\emptyset$, and $G^*[U]$ is complete bipartite.
Let $U'\cup U''$ be the bipartition of $G^*[U]$,
where $|U'|=a\geq2$ and $|U''|=b\geq2$.
Clearly, $G^*$ contains an induced 5-wheel,
whose third-largest eigenvalue equals $1\!-\!\sqrt{5}$.
By Claim \ref{clm-3.2}, if $c\in (\frac14,\frac12]$, then
$G^*$ admits no induced subgraph $H$ satisfying $\sum_{i=2}^{|H|-1}\rho_i^2(H)\geq\frac32.$
Since $(1\!-\!\sqrt{5})^2>\frac32$,
it follows necessarily that $c\in(0,\frac14]$.

As $e(W)=0$ by Claim \ref{clm-3.4},
we have $\mathrm{d}_U(w) \geq 1$ for each $w \in W.$
Let $W'=\{w\in W\mid N_U(w)\subseteq U'\},$
$W''=\{w\in W\mid N_U(w)\subseteq U''\}$ and $W'''=W\setminus(W'\cup W'').$

We first show that $N_U(w)=U'$ for all $w\in W'$ and $N_U(w)=U''$ for all $w\in W''$.
To see this, suppose for contradiction that there exist a vertex $w\in W'$ and two vertices $u_1,u_2\in U'$
such that $u_1\in N_U(w)$ but $u_2\notin N_U(w)$.
Then for any $u_3,u_4\in U''$, the vertex subset $\{u^*,u_1,u_2,u_3,u_4,w\}$ must induce a copy of $M_7$ (see Fig.\,\ref{fig-M5678}),
which contradicts Claim \ref{clm-3.2}.
Therefore, $N_U(w)=U'$ for all $w\in W'$.
An analogous argument shows that $N_U(w)=U''$ for all $w\in W''$.

We next establish that
$|W'|\leq 1$ and $|W''|\leq 1$.
Assume without loss of generality that $|W''|\leq |W'|$.
We argue by contradiction: suppose $|W'|\geq 2$.
Take $w_1,w_2\in W',$ $u_1,u_2\in U'$, and $u_3,u_4\in U''.$
Then, $\{u^*,u_1,u_2,u_3,u_4,w_1,w_2\}$ induces a copy of $M_8$ (see Fig.\,\ref{fig-M5678}), a contradiction.
Therefore, we conclude $|W'|\leq1$ and $|W''|\leq 1.$

We now prove that $N_U(w)=U$ for all $w\in W'''.$
We proceed by contradiction.
Suppose that there exist $w\in W'''$ and $u_0\in U$ such that $u_0w\notin E(G^*).$
Without loss of generality, assume that $u_0\in U'.$
Since $w\notin W'\cup W'',$
there exist $u_1\in U'$ and $u_2\in U''$ with $u_1,u_2\in N_U(w).$
Then, the join of $u_2$ and $u_0u^*u_1w$ is a copy of $M_2$ (see Fig.\,\ref{fig-M1234}), 
yielding a contradiction.
Therefore, $N_U(w)=U$ for all $w\in W'''.$

From the three foregoing observations together with \eqref{eq18}, it follows that
\begin{equation}\label{eq21}
t\big(G^*\big)=e\big(U\big)\big|\{u^*\}\cup W'''\big|
=ab\big(1\!+\!|W'''|\big)\geq 2\big(|U|\!-\!2\big)\big(|W|\!-\!1\big).
\end{equation}
Recall that $e(W)=0$.
In light of (\ref{eq19}) and (\ref{eq20}), we derive that
\begin{equation*}
m=e\big(U\cup\{u^*\}\big)+e\big(U,W\big)
\leq\frac32s+2+|U||W|\leq2s+4+|U|\big(|W|\!-\!1\big),
\end{equation*}
which yields $|W|\!-\!1\geq \frac{1}{|U|}(m\!-\!2s\!-\!4)$.
Recall $|U|\geq\rho_1>\sqrt{m}$ and $s=\big(c\pm o(1)\big)m$.
Combining inequality (\ref{eq21}) with the lower bound for $|W|-1$, we obtain
\begin{align*}
t(G^*)\geq\frac{2|U|\!-\!4}{|U|}\big(m\!-\!2s\!-\!4\big)\geq
\big(2\!-\!o(1)\big)\big(m\!-\!2s\!-\!4\big)=\big(2\!-\!4c\!-\!o(1)\big)m.
\end{align*}
Since $c\in(0,\frac14]$, we have $2-4c>2c$. It follows that $t(G^*)>s$, which contradicts the choice of $G^*$.
Therefore, $G^*[U_1]$ is a star, as claimed.
\end{proof}

Let $u_0$ denote the center of the star $G^*[U_{1}]$,
and let $u_1,u_2,\ldots,u_r$ denote its leaves, with $r=e(U_{1})\geq2\big(c-o(1)\big)\sqrt{m}$
by Claim \ref{clm-3.3}.

\begin{claim}\label{clm-3.7}
If $W\neq\emptyset$, then $N(w)=\{u_1,\ldots,u_r\}$ for all $w\in W$;
moreover, $U_0=\emptyset.$
\end{claim}

\begin{proof}
Recall that $e(W)=0$.
Choose an arbitrary vertex $w\in W$. If $d(w)=1$, then its only neighbor $u$ belongs to $U$.
Now, we define a graph $G=G^*-\{uw\}+\{u^*w\}$. Then, $t(G)=t(G^*)$.
However, by Lemma \ref{lem-2.3}, $\rho(G)>\rho(G^*)$, which contradicts the choice of $G^*$.
Hence, we have $d(w)\geq2$ for any $w\in W$.
In the following, the proof is divided into three parts.

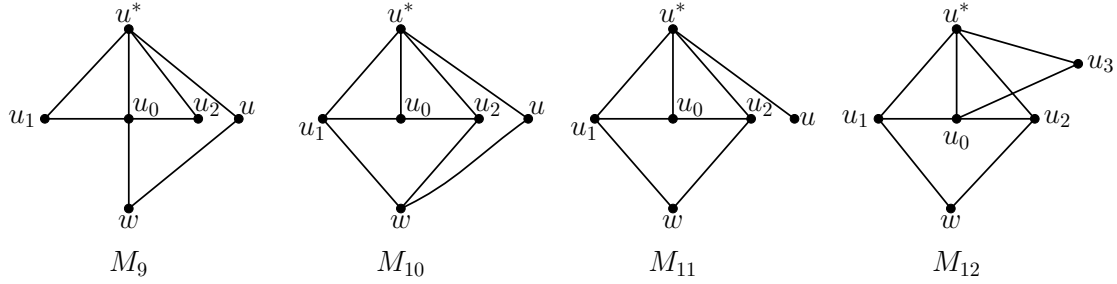
\begin{figure}[htbp]
\centering
\resizebox{0.98\textwidth}{!}{%
\begin{tikzpicture}[
    scale=1.0,
    x=1cm,y=1cm,
every node/.style={inner sep=1.2pt},
v/.style={circle,fill=black,draw=black,inner sep=0pt,minimum size=1.5mm},
lab/.style={font=\large},
sublab/.style={font=\large},
main/.style={draw=black,line width=0.3mm,line cap=round,line join=round}
]

\begin{scope}[shift={(0,0)}]
  \coordinate (T) at (0,1.55);
  \coordinate (W) at (0,-1.55);
  \coordinate (O) at (0,0);
  \coordinate (L) at (-1.45,0);
  \coordinate (R) at (1.2,0);
  \coordinate (U) at (1.9,0);

  \draw[main] (L)--(T)--(R);
  \draw[main] (T)--(U);
  \draw[main] (T)--(W);
  \draw[main] (L)--(R);
  \draw[main] (W)--(U);

  \node[v] at (T) {};
  \node[v] at (W) {};
  \node[v] at (O) {};
  \node[v] at (L) {};
  \node[v] at (R) {};
  \node[v] at (U) {};

  \node[lab,above=2pt] at (T) {$u^*$};
  \node[lab,below=2pt] at (W) {$w$};
  \node[lab,left=3pt]  at (L) {$u_1$};

  \node[lab] at (0.30,0.2) {$u_0$};
  \node[lab] at (1.35,0.2) {$u_2$};
  \node[lab] at (2.05,0.2) {$u$};

  \node[sublab] at (0,-2.5) {$M_9$};
\end{scope}

\begin{scope}[shift={(4.7,0)}]
  \coordinate (T) at (0,1.55);
  \coordinate (W) at (0,-1.55);
  \coordinate (O) at (0,0);
  \coordinate (L) at (-1.35,0);
  \coordinate (R) at (1.35,0);
  \coordinate (U) at (2.2,0);

  \draw[main] (L)--(T)--(R)--(W)--(L);
  \draw[main] (T)--(O);
  \draw[main] (L)--(R);
  \draw[main] (U)--(T);
  \draw[main] (W) to[out=25,in=-145] (U);

  \node[v] at (T) {};
  \node[v] at (W) {};
  \node[v] at (O) {};
  \node[v] at (L) {};
  \node[v] at (R) {};
  \node[v] at (U) {};

  \node[lab,above=2pt] at (T) {$u^*$};
  \node[lab,below=2pt] at (W) {$w$};
  \node[lab,left=-3pt,yshift=-6pt] at (L) {$u_1$};

  \node[lab] at (0.30,0.2) {$u_0$};
  \node[lab] at (1.5,0.2) {$u_2$};
  \node[lab] at (2.35,0.2) {$u$};

  \node[sublab] at (0,-2.5) {$M_{10}$};
\end{scope}

\begin{scope}[shift={(9.4,0)}]
  \coordinate (T) at (0,1.55);
  \coordinate (W) at (0,-1.55);
  \coordinate (O) at (0,0);
  \coordinate (L) at (-1.35,0);
  \coordinate (R) at (1.35,0);
  \coordinate (U) at (2.1,0);

  \draw[main] (L)--(T)--(R)--(W)--(L);
  \draw[main] (T)--(O);
  \draw[main] (L)--(R);
  \draw[main] (T)--(U);

  \node[v] at (T) {};
  \node[v] at (W) {};
  \node[v] at (O) {};
  \node[v] at (L) {};
  \node[v] at (R) {};
  \node[v] at (U) {};

  \node[lab,above=2pt] at (T) {$u^*$};
  \node[lab,below=2pt] at (W) {$w$};
  \node[lab,left=-3pt,yshift=-6pt] at (L) {$u_1$};

  \node[lab] at (0.30,0.2) {$u_0$};
  \node[lab] at (1.5,0.2) {$u_2$};
  \node[lab,right=1pt] at (U) {$u$};

  \node[sublab] at (0,-2.5) {$M_{11}$};
\end{scope}

\begin{scope}[shift={(14.3,0)}]
  \coordinate (T)  at (0,1.55);     
  \coordinate (W)  at (-0.1,-1.55); 
  \coordinate (O)  at (0,0);        
  \coordinate (L)  at (-1.35,0);     
  \coordinate (R)  at (1.35,0);     
  \coordinate (UR) at (2.1,0.95);   

  \draw[main] (L)--(T);
  \draw[main] (T)--(O);
  \draw[main] (T)--(R);
  \draw[main] (T)--(UR);
  \draw[main] (L)--(O);
  \draw[main] (O)--(R);
  \draw[main] (O)--(UR);
  \draw[main] (L)--(W);
  \draw[main] (W)--(R);

  \node[v] at (T) {};
  \node[v] at (W) {};
  \node[v] at (O) {};
  \node[v] at (L) {};
  \node[v] at (R) {};
  \node[v] at (UR) {};

  \node[lab,above=2pt] at (T) {$u^*$};
  \node[lab,below=2pt] at (W) {$w$};
  \node[lab,left=3pt] at (L) {$u_1$};
  \node[lab,below=5pt] at (O) {$u_0$};
  \node[lab,right=3pt] at (R) {$u_2$};
  \node[lab,right=3pt] at (UR) {$u_3$};

  \node[sublab] at (0,-2.5) {$M_{12}$};
\end{scope}

\end{tikzpicture}%
}
\caption{The forbidden induced subgraphs $M_9$, $M_{10}$, $M_{11}$, and $M_{12}$.}
\label{fig-M9-M12}
\end{figure}

\begin{table}[htbp]
\centering
\caption{The intermediate eigenvalues of $M_9$, $M_{10}$, $M_{11}$, and $M_{12}$.}
\label{tab-M9101112}
\begin{tabular*}{\textwidth}{@{\extracolsep{\fill}}ccccc@{}}
\toprule
& $M_9$ & $M_{10}$ & $M_{11}$ & $M_{12}$ \\
\midrule
$\rho_2$ & 0.8061 & 0.7020 & 0.6648 & 0.8347 \\
$\rho_3$ & 0.0000 & 0.0000 & 0.0000 & 0.0000 \\
$\rho_4$ & 0.0000 & 0.0000 & 0.0000 & -0.6272 \\
$\rho_5$ & -1.7093 & -1.2855 & -1.3684 & -1.0000 \\
\bottomrule
\end{tabular*}
\end{table}

To begin with, we show that $N(w)=U$ for all $w\in W$ satisfying $w\sim u_0$.
Suppose to the contrary that there exists $w\in W$ such that $w\sim u_0$ yet $N(w)\subsetneq U$.
If $w$ has no neighbors in $\{u_1,\ldots,u_r\}$,
then $w$ must have a neighbor $u\in U_0$. Consequently,
the vertex subset $\{u^*,u_0,u_1,u_2,u,w\}$ induces a copy of $M_9$ (see Fig.\,\ref{fig-M9-M12}), yielding a contradiction.
If $w$ is adjacent to all vertices in $\{u_1,\ldots,u_r\}$, the proper inclusion $N(w)\subsetneq U$
guarantees the existence of some $u\in U_0$ with $w\nsim u$.
In this case, $\{u^*,u_0,u_1,u_2,u,w\}$ induces a copy of $M_7$ (see Fig.\,\ref{fig-M5678}), another contradiction.
The only remaining possibility is that $w$ has both a neighbor $u_1$ and a non-neighbor $u_2$ in $\{u_1,\ldots,u_r\}$.
Under this setup, the join of $u_0$ and $u_2u^*u_1w$ induces a copy of $M_2$ (see Fig.\,\ref{fig-M1234}),
which again gives a contradiction.
Therefore, we conclude that $N(w)=U$ for all $w\in W$ satisfying $w\sim u_0$.

We next prove that $N(w)=U$ for every $w\in W$ with $N(w)\cap U_0\neq \emptyset$.
Suppose, for contradiction, that some vertex $w \in W$ satisfies $N(w)\cap U_0\neq \emptyset$,
i.e., $w$ has a neighbor $u \in U_0$.
Based on our preceding conclusion,
if $w \sim u_0$, then $N(w)=U$ holds immediately.
We may thus assume $w \nsim u_0$.
We claim that $w$ is adjacent to each vertex in $\{u_1,\dots,u_r\}$,
Indeed, if $w\nsim u_i$ for some $i\in\{1,\ldots,r\}$, then $\{u_0u_i, uw\}$ induces a copy of $2K_2$, a contradiction.
However, under this adjacency configuration,
$\{u^*, u_0, u_1, u_2, u, w\}$ induces a copy of $M_{10}$ (see Fig.\,\ref{fig-M9-M12}),
again a contradiction.
Hence, we conclude that $N(w)=U$ holds for all $w\in W$ with $N(w)\cap U_0\neq \emptyset$.

Finally, we address the remaining case:
assume there exists some $w \in W$ with $N(w)\subseteq\{u_1,\dots,u_r\}$.
Under this assumption, we shall show $U_0=\emptyset$ and $N(w)=\{u_1,\dots,u_r\}$.
First, we show $U_0=\emptyset$. Since $d_U(w)=d(w)\geq2$,
we can select two vertices $u_1,u_2\in N_U(w)$.
If $U_0 \neq \emptyset$, choose any $u\in U_0$.
Then, the vertex subset $\{u^*, u_0, u_1, u_2, u, w\}$ induces a copy of $M_{11}$ (see Fig.\,\ref{fig-M9-M12}),
which gives a contradiction.
Hence, $U_0=\emptyset$.
Next, we prove $N(w)=\{u_1,\dots,u_r\}$.
Suppose that
there exists a vertex $u_3\in\{u_1,\dots,u_r\}$ with $w \nsim u_3$.
In this case, $\{u^*,u_0,u_1,u_2,u_3,w\}$ induces a copy of $M_{12}$ (see Fig.\,\ref{fig-M9-M12}),
another contradiction. Hence, $N(w)=\{u_1,\dots,u_r\}$.

From the three preceding conclusions,
every $w\in W$ satisfies either $N(w)=U$ or
$N(w)=\{u_1,\ldots,u_r\}$.
Suppose both cases are realized simultaneously,
so there exist $w_1,w_2\in W$ such
that $N(w_1)=U$ and $N(w_2)=\{u_1,\ldots,u_r\}.$
Moreover, $U_0=\emptyset$. As
$w_1\in N(u_0)\setminus N(u^*)$, we obtain $N[u^*]\subsetneq N[u_0]$,
which implies $x_{u^*}<x_{u_0}$, a contradiction.
Hence, the two possibilities cannot hold simultaneously.

Suppose now that $N(w)=U$ for all $w\in W$.
Then, $N(w)=N(u^*)$ and $x_w=x_{u^*}$ for every $w\in W$.
Furthermore, $\rho_1x_u=\sum_{v\in N(u)}x_v=(1+|W|)x_{u^*}$ for every $u\in U_0$.
In this case, if there exist two vertices $u,u'\in U_0,$
then $\{u^*,u_0,u_1,u_2,u,u',w\}$ induces a copy of $M_8$ (see Fig.\,\ref{fig-M5678}),
leading to a contradiction.
Hence, $|U_0|\leq 1$.
Thus,
$$\rho_1x_{u^*}=x_{u_0}+\sum_{u\in U_0}x_u
+\sum_{i=1}^{r}x_{u_i}\leq x_{u_0}+\frac{1\!+\!|W|}{\rho_1}x_{u^*}+\sum_{i=1}^{r}x_{u_i}.$$
On the other hand, we have
$\rho_1x_{u_0}=x_{u^*}+|W|x_{u^*}+\sum_{i=1}^{r}x_{u_i}.$
Under the assumption that $W\neq \emptyset$,
we readily get $x_{u_0}>x_{u^*}$, which gives a contradiction.
Therefore, we conclude $N(w)=\{u_1,\ldots,u_r\}$ for all $w\in W$,
and this forces $U_0=\emptyset$.
\end{proof}

\begin{claim}\label{clm-3.8}
$G^*\cong S_{m,s}^{+}$.
\end{claim}

\begin{proof}
We first demonstrate that $U_0=\emptyset$.
By Claim \ref{clm-3.7}, the desired conclusion holds provided $W\neq\emptyset$.
We hence suppose $W=\emptyset$ while $U_0\neq\emptyset$. 
Since $G^*[U_1]\cong K_{1,r}$, we obtain $m=2r+1+|U_0|$ and $t(G^*)=r$.
If there exist two distinct vertices $u,u'\in U_0$, then 
$\{u^*,u_0,u_1,u_2,u,u'\}$ induced a copy of $M_5$ with dominating vertex $u^*$ (see Fig.\,\ref{fig-M5678}),
yielding a contradiction.
Thus, $|U_0|=1$, which gives $t(G^*)=r=\frac12(m-2).$
Recall that $t(G^*)\leq s=\big(c\pm o(1)\big)m$, where $c\in(0,\frac12]$ is a constant. 
This forces $c=\frac12$.
Write $U_0=\{u\}$. 
Then, $\{u^*,u_0,u_1,u_2,u\}$ induces a dart graph $H=K_1\nabla(P_3\cup K_1)$,
whose fourth-largest eigenvalue satisfies $\rho^2_4(H)>(-1.27)^2>\frac32$.
Since $c=\frac12$, this contradicts Claim \ref{clm-3.2}.
Hence, we have $U_0=\emptyset$.

Combining Claims \ref{clm-3.6} and \ref{clm-3.7}, 
we conclude that the vertex set of $G^*$ admits a partition:
$\{u_1,\ldots,u_r\}$ and $\{u^*,u_0\} \cup W.$
All possible edges between the two parts are present, and moreover, 
$u^*u_0$ is the only edge with both endpoints in the same part. 
Hence, $G^* \cong K_{r,|W\cup\{u_0,u^*\}|} +\{u_0 u^*\},$
i.e., $G^*$ is isomorphic to $S_{m,r}^{+}.$  
A straightforward calculation shows that $\rho_1(S_{m,r}^{+})$ is the largest real root of the equation
\begin{equation*}
x^3-x^2-(m-1)x+(m-1-2r)=0.
\end{equation*}
For $x>1$, the above equation is algebraically equivalent to the following form:
\begin{equation*}
x^2=m-1+\frac{2r}{x-1}.
\end{equation*}

It remains to verify that $r=s$.
From the structure of $S_{m,r}^{+}$, we immediately have $t(G^*)=r$.
Together with the fact that $t(G^*)\leq s$,
this yields $r\leq s$.
On the other hand, the definition of $G^*$ indicates $\rho_1^2\geq m-1+\frac{2s}{\rho_1-1}.$
Therefore, we conclude that $r=s$, completing the proof.
\end{proof}

\end{document}